\documentclass[11pt]{amsart}

\usepackage{amsmath}
\usepackage{amssymb}
\usepackage{graphicx,color}

\newtheorem{theorem}{Theorem}[section]
\newtheorem{corollary}[theorem]{Corollary}
\newtheorem{lemma}[theorem]{Lemma}
\newtheorem{proposition}[theorem]{Proposition}
\theoremstyle{definition}
\newtheorem{definition}[theorem]{Definition}

\newtheorem{remark}[theorem]{Remark}
\newtheorem{Assumptions}{Hypothesis}[section]
\def\ds{\displaystyle}
\def\fh{\mathfrak{h}}
\def\fd{\mathfrak{d}}
\def\fg{\mathfrak{g}}
\numberwithin{equation}{section}

\title[Controllability for a degenerate population equation]{Controllability for a population equation with interior degeneracy}

\author[G. Fragnelli]{Genni Fragnelli}

\address[G. Fragnelli]{Dipartimento di Matematica\\ Universit\`{a} di Bari Aldo Moro, Via
E. Orabona 4, 70125 Bari - Italy}
\email{{\tt genni.fragnelli@uniba.it}}



\begin{document}

\begin{abstract}
We deal with a degenerate model in divergence form describing the dynamics of a population depending
on time, on age and on space. We assume that the degeneracy occurs in the interior of the spatial domain and we focus on null controllability. To this aim, first we prove Carleman estimates for the associated adjoint
problem, then, via cut off functions, we prove the existence of a null control function
localized in the interior of the space domain. We consider two cases: either the control region contains the degeneracy point $x_0$, or the control region is the union of two intervals each of them lying on one side of $x_0$. This paper complement some previous results, concluding the study of the subject.
\end{abstract}

\maketitle

\centerline{\emph{Dedicated to Irena, on the occasion of her 70th birthday,}}
\centerline{\emph{with sincere  esteem}}

\section{Introduction}

We consider the following degenerate population model describing the dynamics of a single species:
\begin{equation} \label{1}
\begin{cases}
\displaystyle {\frac{\partial y}{\partial t}+\frac{\partial y}{\partial a}}
-  (ky_x)_x+\mu(t, a, x)y =f(t,a,x)\chi_{\omega} & \quad \text{in } Q,\\
  y(t, a, 1)=y(t, a, 0)=0 & \quad \text{on }Q_{T,A},\\
 y(0, a, x)=y_0(a, x) &\quad \text{in }Q_{A,1},\\
 y(t, 0, x)=\int_0^A \beta (a, x)y (t, a, x) da  &\quad  \text{in } Q_{T,1}
\end{cases}
\end{equation}
where  $Q:=(0,T)\times(0,A)\times(0,1)$, $Q_{T,A} := (0,T)\times (0,A)$, $Q_{A,1}:=(0,A)\times(0,1)$ and
$Q_{T,1}:=(0,T)\times(0,1)$. Here $y(t,a,x)$  is the distribution of certain individuals  at location $x \in (0,1)$, at time $t\in(0,T)$, where $T$ is fixed, and of age $a\in (0,A)$. $A$ is the maximal age of life, while $\beta$ and $\mu$ are the natural fertility and the natural death rate, respectively. Thus, the formula $\int_0^A \beta y da$ denotes the distribution of newborn individuals at time $t$ and location $x$.  The function $k$, which is the dispersion coefficient, depends on the space variable $x$ and we assume that it degenerates in an interior point $x_0$ of the state space.
In particular, we say that
 
 \begin{definition}\label{Ass0} The function $k$ is
{\bf weakly degenerate  (WD)} if there exists $x_0 \in (0,1)$ such
that $k(x_0)=0$, $k>0$ on $[0, 1]\setminus \{x_0\}$, $k\in
W^{1,1}(0,1)$ and there exists $M\in (0,1)$ so that $(x-x_0)k'
\le Mk$ a.e. in $[0,1]$.
\end{definition}

\begin{definition}\label{Ass01}  The function $k$ is
{\bf strongly degenerate (SD)} if there exists $x_0 \in (0,1)$
such that $k(x_0)=0$, $k>0$ on $[0, 1]\setminus \{x_0\}$, $k\in
W^{1, \infty}(0,1)$ and there exists $M \in [1,2)$ so that
$(x-x_0)k' \le M k$ a.e. in $[0,1]$.
\end{definition}

 For example, as $k$ one can consider $k(x)=|x-x_0|^\alpha$, $\alpha >0$. 

Finally, in the model, $ \chi_\omega$ is the characteristic function of  the control region $\omega\subset(0,1)$ which can contain $x_0$ or can be the union of two intervals each of them lying on different sides of the degeneracy point, more precisely:
\[\omega= \omega_1 \cup \omega_2\]
where
\[
\omega_i= (\lambda_i , \beta_i ) \subset (0, 1), i = 1, 2, \text{ and } \beta_1 < x_0 < \lambda_2.
\]

It is known that the asymptotic behavior of the solution for  the Lotka-McKendrick system depends on the so called net reproduction rate $R_0$: indeed the solution can be exponentially growing if $R_0> 1$, exponentially decaying if $R_0<1$ or  tends to the steady state solution if $R_0=1$.   Clearly, if the system represents the distribution of a damaging insect population or of a pest population and $R_0>1$, it is very worrying. For this reason, recently great attention is given to null controllability issues. For example in \cite{he}, where \eqref{1} models an insect growth, the control corresponds to a removal of individuals by using pesticides. If
{\it $k$ is a constant or a strictly positive function}, null controllability for \eqref{1} is studied, for example, in \cite{An}.  If {\it $k$ degenerates at the boundary or at an interior point of the domain} and  $y$ is independent of $a$ we refer, for example, to \cite{acf}, \cite{fm}, \cite{fm1} and to \cite{fm2}, \cite{fm_hatvani}, \cite{fm_opuscola} if $\mu$ is singular at the same point of $k$. Actually,
\cite{aem} is the first paper where  $y$ depends on $t$, $a$ and $x$ and the dispersion coefficient $k$ degenerates.  In particular, in \cite{aem}, $k$ degenerates  at the boundary of the domain (for example $k(x) = x^\alpha,$  being $x \in (0,1)$ and $\alpha >0$). Using Carleman estimates for the adjoint problem, the authors prove null controllability for \eqref{1} under the condition $T\ge A$. The case $T<A$ is considered  in \cite{idriss}, \cite{em}, \cite{f_anona} and \cite{fJMPA}. In \cite{em} the problem is always in {\it divergence form} and the authors assume that $k$ degenerates only at a point of the boundary; moreover, they use
the fixed point technique in which
the birth rate $\beta$ must be of class $C^2(Q)$ (necessary requirement in the proof of \cite [Proposition 4.2]{em}). A more general result is obtained in \cite{f_anona} where $\beta$ is only a continuous function, but $k$ can degenerate at both extremal points. In \cite{idriss} the problem is in {\it divergence form} and $k$ degenerates at an interior point and it belongs to $C[0,1]\cap C^1([0,1]\setminus\{x_0\})$. Finally, in \cite{fJMPA}, we studied null controllability for \eqref{1} in {\it non divergence form} and with a diffusion coefficient degenerating at a one point of the boundary domain or in an interior point. In this paper we study the null controllability for \eqref{1} assuming that $k$  degenerates at $x_0 \in (0,1)$ and $T<A$ or $T >A$. We underline that here, contrary to \cite{idriss}, the 
function $k$ is less regular,  the control region $\omega$ not only can contain $x_0$, but can be also the union of two intervals each of them lying on one side of $x_0$ and $T$ can be greater than $A$. Moreover, contrary to \cite{aem}, where $T>A$ and $k$ degenerates at the boundary, here we assume that $T$ can be smaller than $A$ and $k$ degenerates at $x_0 \in (0,1)$.
Hence, this paper is the completion of all the previous ones. Moreover, the technique used in Theorem \ref{CorOb1'} can be also applied either when $k$ degenerates at the boundary of the domain, completing \cite{f_anona}, or when $k$ is in non divergence form and $k$ degenerates at the boundary or in the interior of the domain, completing \cite{fJMPA}.
Finally, observe that in this paper, as in \cite{f_anona} or in \cite{fJMPA}, we do not consider the positivity of the solution, even if it is clearly an interesting question to face: this problem is related to the minimum time, i.e. $T$ cannot be too small (see \cite{zuazua} for related results in non degenerate cases. This topic will be the subject of further investigations.

A final comment on the notation: by $c$ or $C$ we shall denote
{\em universal} strictly positive constants, which are allowed to vary from line to
line.

\section{Well posedness results}\label{sec3-1}
For the well posedness of the problem,
we assume the following hypotheses on the rates $\mu$ and $\beta$ :
\begin{Assumptions}\label{ratesAss}
 The  functions $\mu$ and $\beta$ are such that
\begin{equation}\label{3}
\begin{aligned}
&\bullet \beta \in C(\bar Q_{A,1}) \text{ and } \beta \geq0  \text{ in } Q_{A,1}, \\
&\bullet \mu \in C(\bar Q) \text{ and }  \mu\geq0\text{ in } Q.
\end{aligned}
\end{equation}
\end{Assumptions}
To prove well possessedness  of  \eqref{1}, we introduce, as in \cite{fm1}, the following
 Hilbert spaces
\[
\begin{aligned}
 H^1_{k} (0,1):=\Big\{ u \in W^{1,1}_0(0,1) \,:\, \sqrt{k} u' \in  L^2(0,1)\Big\}
\end{aligned}
\]
and
\[
H^2_k :=  \{ u \in H^1_k(0,1) |\,ku_x \in
H^1(0,1)\}. 
\]
 We have, as in \cite{fm1}, that the  operator
\[\mathcal A_0u:= (ku_{x})_x,\qquad    D(\mathcal A_0): = H^2_{k}(0,1)\]
is self--adjoint, nonpositive  and generates an analytic contraction
semigroup of angle $\pi/2$ on the space $L^2(0,1)$.

As in \cite{f_anona}, setting $ \mathcal A_a u := \ds \frac{\partial  u}{\partial a}$, we have that
\[
\mathcal Au:= \mathcal A_a u - 
\mathcal A_0 u,
\]
for 
\[
\begin{aligned}
u \in D(\mathcal A) =&\left\{u \in L^2(0,A;D(\mathcal A_0)) : \frac{\partial u}{\partial a} \in  L^2(0,A;H^1_k(0,1)), \right. \\&
    \left.\quad u(0, x)= \int_0^A \beta(a, x) u(a, x) da\right\},
\end{aligned}
\]
generates a strongly continuous semigroup on $L^2(Q_{A,1}):= L^2(0,A; L^2(0,1))$ (see also \cite{iannelli}). Moreover, the operator $B(t)$ defined as
\[
B(t) u:= \mu(t,a,x) u,
\]
for $u \in D(\mathcal A)$, can be seen as a bounded perturbation of $\mathcal A$ (see, for example, \cite{acf}); thus also
$
(\mathcal A + B(t), D(\mathcal A))
$ generates a strongly continuous semigroup.

Setting $L^2(Q):= L^2(0,T;L^2(Q_{A,1}))$, the following well posedness  result holds (see \cite{f_anona} for the proof):
\begin{theorem}\label{theorem_existence}
Assume that $k$ is weakly or strongly degenerate at $0$ and/or at $1$.  For all $f \in
L^2(Q)$ and $y_0 \in L^2(Q_{A,1})$, the system \eqref{1} admits a unique solution 
\[y \in \mathcal U:= C\big([0,T];
L^2(Q_{A,1}))\big) \cap L^2 \big(0,T;
H^1(0,A; H^1_k(0,1))\big)\]
and 
\begin{equation}\label{stimau}
\begin{aligned}
\sup_{t \in [0,T]} \|y(t)\|^2_{L^2(Q_{A,1})} &+\int_0^T\int_0^A\|\sqrt{k}y_x\|^2_{L^2(0,1)}dadt \\
&\le C \|y_0\|^2_{L^2(Q_{A,1})}  + C\|f\|^2_{L^2(Q)},
\end{aligned}
\end{equation}
where $C$ is a positive constant independent of $k, y_0$ and $f$.

In addition, if $f\equiv 0$,  then
$
y\in C^1\big([0,T];L^2(Q_{A,1})\big).
$ 
\end{theorem}
\section{Carleman estimates}\label{sec-3}

In this section we show degenerate 
Carleman estimates for the following adjoint system associated to \eqref{1}:
\begin{equation} \label{adjoint}
\begin{cases}
\ds \frac{\partial z}{\partial t} + \frac{\partial z}{\partial a}
+(k(x)z_{x})_x-\mu(t, a, x)z =f ,& (t,a,x) \in Q,\\
  z(t, a, 0)=z(t, a, 1)=0, & (t,a) \in Q_{T,A},\\
   z(t,A,x)=0, & (t,x) \in Q_{T,1}.
\end{cases}
\end{equation}

On  $k$ we make additional assumptions:
\begin{Assumptions}\label{BAss01} The function $k$ is {\bf(WD)} or {\bf (SD)}.  Moreover,  if $M > \ds \frac{4}{3}$, then 
there exists a constant $\theta \in \left(0, M\right]$ such that
\begin{equation}\label{dainfinito_1}
\begin{array}{ll}
x \mapsto \dfrac{k(x)}{|x-x_0|^{\theta}} &
\begin{cases}
& \mbox{ is non increasing on the left of $x=x_0$,}\\
& \mbox{ is non decreasing on the right of $x=x_0$}.
\end{cases}
\end{array}
\end{equation}
In addition, when $ M >\displaystyle \frac{3}{2}$ the function in  \eqref{dainfinito_1}  is bounded below away from $ 0$
and there exists a constant $\Gamma >0$ such that
\begin{equation}\label{Sigma}
|k'(x)|\leq \Gamma |x-x_0|^{2\theta-3} \mbox{ for a.e. }x\in
[0,1].
\end{equation}
\end{Assumptions}
Now, let us introduce the weight function

\begin{equation}\label{13}
\varphi(t,a,x):=\Theta(t,a)\psi(x),
\end{equation}
where 
\begin{equation}\label{theta}
\Theta(t,a):=\frac{1}{[t(T-t)]^4a^4}\quad \text{and}\quad
\psi(x) :=
c_1\left[\int_{x_0}^x \frac{y-x_0}{k(y)}dy- c_2\right].
\end{equation}
The following estimate holds:

\begin{theorem}\label{Cor1}
Assume that Hypothesis $\ref{BAss01}$ is satisfied. Then,
there exist two strictly positive constants $C$ and $s_0$ such that every
solution $v$ of \eqref{adjoint} in
\[
\mathcal{V}:=L^2\big(Q_{T,A}; H^2_k(0,1)\big) \cap H^1\big(0,T; H^1(0,A;H^1_k(0,1))\big)\]
satisfies, for all $s \ge s_0$,
\[
\begin{aligned}
&\int_{Q} \left(s\Theta k (v_x)^2 + s^3 \Theta^3
\frac{(x-x_0)^2}{k}v^2\right)e^{2s\varphi}dxdadt\\
&\le C\left(\int_{Q} f^{2}e^{2s\varphi}dxdadt +
sc_1\int_0^T\int_0^A\left[k\Theta e^{2s\varphi}(x-x_0)(v_x)^2
dadt\right]_{x=0}^{x=1}dadt\right)
\end{aligned}\]
\end{theorem}

Clearly the previous  Carleman estimate holds
for every function $v$ that satisfies \eqref{adjoint} in $(0,T)\times(0,A)\times (B,C)$ as long as $(0,1)$ is substituted by $(B,C)$ and $k$ satisfies Hypothesis \ref{BAss01} in $(B,C)$.

\vspace{0.4cm}

\paragraph{Proof of Theorem \ref{Cor1}}
The proof of Theorem \ref{Cor1} follows the ideas of the one of
\cite[Theorem 3.1]{f_anona} or \cite[Theorem 3.6]{fJMPA} (for the non divergence case). As in the previous papers, we consider, first of all, the case when $\mu\equiv 0$: for
every $s> 0$ consider the function
\[
w(t,a, x) := e^{s \varphi(t,a, x)}v(t,a, x),
\]
where $v$ is any solution of \eqref{adjoint} in $\mathcal{V}$, so that
also $w\in\mathcal{V}$, since $\varphi<0$.
Moreover, $w$ satisfies 
\begin{equation}\label{1'}
\begin{cases}
(e^{-s\varphi}w)_t + (e^{-s\varphi}w)_a +(k (e^{-s\varphi}w)_x)_{x}  =f(t,a,x), & (t,x) \in
Q,
\\[5pt]
w(0, a, x)= w(T,a, x)= 0, & (a,x) \in Q_{A,1},
\\[5pt]
w(t,A,x)=w(t,0,x)=0, & (t,x) \in  Q_{T,1},
\\[5pt]
w(t, a,0)= w(t, a, 1)= 0, & (t,a) \in  Q_{T,A}.
\end{cases}
\end{equation} and \cite[Lemma 3.1]{f_anona} still holds.
In particular, setting 
\[
\begin{cases}
L^+_sw :=( kw_{x})_x
 - s (\varphi_t+ \varphi_a) w + s^2k \varphi_x^2 w,
\\[5pt]
L^-_sw := w_t + w_a-2sk\varphi_x w_x -
 s(k\varphi_{x})_xw,
 \end{cases}
\]we have
\begin{lemma}\label{lemma1}[see \cite[Lemma 3.1]{f_anona}] Assume Hypothesis $\ref{BAss01}$.
The following identity holds
\begin{equation}\label{D&BT}
\left.
\begin{aligned}
<L^+_sw,L^-_sw>_{L^2(Q)}
\;&=\;
\frac{s}{2} \int_Q(\varphi_{tt}+\varphi_{aa}) w^2dxdadt \\
&+ s \int_Qk(x) (k(x)
\varphi_x)_{xx} w w_xdxdadt
\\&- 2s^2 \int_Qk \varphi_x \varphi_{tx}w^2dxdadt - 2s^2\int_{Q}k \varphi_x\varphi_{xa}w^2dxdadt\\
&+s
\int_Q(2 k^2\varphi_{xx} + kk'\varphi_x)w_x^2 dxdadt
\\
&+ s^3 \int_Q(2k \varphi_{xx} + k'\varphi_x)k
\varphi^2_x w^2dxdadt \\
&+s\int_{Q}\varphi_{at}w^2 dxdadt.
\end{aligned}\right\}\;\text{\{D.T.\}}
\end{equation}
\begin{equation}\nonumber
\hspace{55pt}
\text{\{B.T.\}}\;\left\{
\begin{aligned}
&
\int_{Q_{T,A}}[kw_xw_t]_{0}^{1} dadt+\int_{Q_{T,A}}\big[kw_xw_a\big]_0^{1}dadt\\& -\frac{s}{2}\int_{Q_{A,1}} \left[\varphi_a w^2\right]_0^T dxda.\\& +
\int_{Q_{T,A}}[-s\varphi_x (k(x)w_x)^2 +s^2k(x)\varphi_t \varphi_x w^2\\& -
s^3 k^2\varphi_x^3w^2 ]_{0}^{1}dadt\\
& +
\int_{Q_{T,A}}[-sk(x)(k(x)\varphi_x)_xw w_x]_{0}^{1}dadt\\&+ s^2 \int_{Q_{T,A}}\big[k\varphi_x\varphi_aw^2\big]_0^{1}dadt\\[3pt]&
-\frac{1}{2}\int_{Q_{T,1}} \big[kw_x^2\big]_0^Adxdt
+\frac{1}{2}\int_{Q_{T,1}}\big[ \big(s^2k \varphi_x^2 \\&- s (\varphi_t+\varphi_a) \big)w^2\big]_0^Adxdt.\end{aligned}\right.
\end{equation}
\end{lemma}
 We underline the fact that 
in this case all integrals and integrations by parts are justified by the
definition of $D(\mathcal A)$ and the choice of $\varphi$, while, if the degeneracy is at the boundary of the domain as in \cite{f_anona}, they were guaranteed by the choice of Dirichlet
conditions at $x=0$ or $x=1$, i.e. where the operator is degenerate.

As a consequence of the definition of $\varphi$, one has the next estimate:
\begin{lemma}\label{lemma2}Assume Hypothesis $\ref{BAss01}$. There exist two  strictly positive constants $C$ and $s_0$ such
that, for all $s\ge s_0$,  all solutions $w$ of \eqref{1'}
satisfy the following estimate
\[
sC\int_{Q}\Theta k w_x^2 dxdadt
+s^3C\int_{Q}\Theta^3 \frac{(x-x_0)^2}{k}w^2 dxdadt \le \big\{D.T.\big\} .
\]
\end{lemma}
\begin{proof}
Using the definition of $\varphi$, the distributed terms given in Lemma \ref{lemma1}
take the form
\[
\text{\{D.T.\}}=\;\left\{\begin{aligned}
&\frac{s}{2}\int_{Q}(\Theta_{tt}+ \Theta_{aa})\psi w^2 dxdadt
-2s^2c_1\int_{Q}\Theta{\Theta_t}\frac{(x-x_0)^2}{k}w^2 dxdadt\\
&-2s^2c_1\int_{Q}\Theta{\Theta_a}\frac{(x-x_0)^2}{k}w^2 dxdadt\\
&+ sc_1 \int_{Q}\Theta \left(2-\frac{k'}{k}
(x-x_0)\right)k(w_x)^2 dxdadt\\
&+ s^3c_1^3\int_{Q}\Theta^3\left(2-\frac{k'}{k}
(x-x_0)\right)\frac{(x-x_0)^2}{k}w^2dxdadt\\
&+ s  \int_{Q}\Theta_{ta}\psi w^2 dxdadt.
\end{aligned}\right.
\]
Because of the choice of $\varphi(x)$, one has, as in \cite{fm1}, 
\[
  2-\frac{(x-x_0)k'}{k}\ge 2-M \quad \text{a.e. } \; x\in[0,1].
\]
Thus, there exists $C>0$ such that, the distributed terms satisfy the estimate
\begin{equation}\label{aaaaa}
\begin{aligned}
\{D.T.\} &\ge\frac{s}{2}\int_{Q}(\Theta_{tt}+ \Theta_{aa})\psi w^2 dxdadt -s^2C\int_{Q}|\Theta\Theta_t|\frac{(x-x_0)^2}{k}w^2
dxdadt\\
&-s^2C\int_{Q}|\Theta\Theta_a|\frac{(x-x_0)^2}{k}w^2
dxdadt\\
&+ s C\int_{Q}\Theta (w_x)^2 dxdadt+ s^3C\int_{Q}\Theta^3\frac{(x-x_0)^2}{k}w^2
dxdadt\\
&+s\int_{Q}\Theta_{ta}\psi w^2 dxdadt.
\end{aligned}
\end{equation}
By \cite[Lemma 3.5]{fJMPA}, we conclude that, for $s$ large enough,
\[
\begin{aligned}
s^2C\int_{Q}(|\Theta\Theta_t|+|\Theta \Theta_a|)\frac{(x-x_0)^2}{k}
 w^2 dxdadt&\le C s^2
\int_{Q}\Theta^3\frac{(x-x_0)^2}{k}w^2 dxdadt\\
&\le \frac{C^3}{4}s^3\int_{Q}\Theta^3 \frac{(x-x_0)^2}{k}w^2 dxdadt.
\end{aligned}
\]
Again as in \cite[Lemma 4.1]{fm1}, we get
\begin{equation}\label{quasfin}
\begin{aligned}
\left|  \frac{s}{2}\int_{Q}(\Theta_{tt} + \Theta_{aa})\psi
w^2dxdadt \right| 
&\leq sC\int_Q\Theta ^{3/2} w^2 dxdadt
\\&
\le \frac{C}{4}s\int_{Q} \Theta  k (w_x)^2
dxdadt
\\&+ \frac{C^3}{4}s^3\int_{Q}\Theta^3
\frac{(x-x_0)^2}{k}w^2 dxdadt.
\end{aligned}
\end{equation}
Analogously, one has that the last term in \eqref{aaaaa}, i.e. $s\int_{Q} \Theta_{ta}\psi w^2 dxdadt$ satisfies
\[
\begin{aligned}
 \left|s\int_{Q} \Theta_{ta}\psi w^2 dxdadt\right| &\le  \frac{C}{4}s\int_{Q} \Theta  k (w_x)^2
dxdadt \\
&+   \frac{C^3}{4}s^3\int_{Q}\Theta^3
\frac{(x-x_0)^2}{k}w^2 dxdadt.
\end{aligned}
\]
Summing up, we obtain
\[
\begin{aligned}
\{D.T.\}&\ge -\frac{C}{4}s\int_{Q} \Theta  (w_x)^2 dxdadt -
\frac{C^3}{4}s^3\int_{Q}\Theta^3
\left(\frac{x-x_0}{k}\right)^2w^2 dxdadt \\
& -\frac{C^3}{4}s^3\int_{Q}\Theta^3 \left(\frac{x-x_0}{k}
\right)^2w^2 dxdadt
\\&
+ s C\int_{Q}\Theta (w_x)^2 dxdadt+ s^3C\int_{Q}\Theta^3\left(\frac{x-x_0}{k} \right)^2w^2 dxdadt\\&
-\frac{C}{4}s\int_{Q} \Theta  (w_x)^2 dxdadt -
\frac{C^3}{4}s^3\int_{Q}\Theta^3(w_x)^2 dxdadt
\\&
\ge \frac{C}{4}s\int_{Q} \Theta (w_x)^2 dxdadt +\frac{C^3}{4}s^3
\int_{Q}\Theta^3 \left(\frac{x-x_0}{k} \right)^2 w^2 dxdadt.
\end{aligned}
\]
\end{proof}

Proceeding as in \cite{f_anona} and in \cite{fm1}, one has for the boundary terms the following lemma:
\begin{lemma}\label{lemma41}
Assume Hypothesis $\ref{BAss01}$. The boundary terms in
\eqref{D&BT} reduce to
\[-sc_1\int_0^T\int_0^A\Theta(t)k\Big[(x-x_0)(w_x)^2\Big]_{x=0}^{x=1}dadt.
\]
\end{lemma}

By Lemmas \ref{lemma1}-\ref{lemma41}, there exist $C>0$ and
$s_0>0$ such that all solutions $w$ of \eqref{1'} satisfy, for
all $s \ge s_0$,
\[
\begin{aligned}
& s\int_{Q}\Theta k w_x^2 dxdadt
+s^3\int_{Q}\Theta^3 \frac{(x-x_0)^2}{k}w^2 dxdadt \\
&\le C\left(\int_{Q} f^2 e^{2s\varphi}dxdadt+
sc_1\int_0^{T}\int_0^A \left[\Theta k(x)
(x-x_0)(w_x)^2\right]_{x=0}^{x=1}dadt \right).
\end{aligned}
\]

Hence, if $\mu \equiv 0$, Theorem \ref{Cor1} follows recalling the definition of $w$ and the fact that
\[
L^+_sw + L^-_sw=e^{s\varphi}f,
\]

If $\mu \not \equiv 0$, we consider the function $\overline{f}=f+\mu v$.
 Hence,  there are  two  strictly positive constants $C$ and $s_0$ such that, for
all $s\geq s_0$, the following inequality holds
\begin{equation} \label{fati1?}
\begin{aligned}
&\int_{Q} \left(s\Theta k (v_x)^2 + s^3 \Theta^3
\frac{(x-x_0)^2}{k}v^2\right)e^{2s\varphi}dxdadt\\
&\le C\left(\int_{Q} \bar f^{2}e^{2s\varphi}dxdadt +
s\int_0^T\int_0^A\left[k\Theta e^{2s\varphi}(x-x_0)(v_x)^2
dadt\right]_{x=0}^{x=1}dadt\right).
\end{aligned}
\end{equation}
On the other hand, we have
\begin{equation} \label{4'?}
\begin{aligned}
\int_{Q}\overline{f}^{2}e^{2s\varphi}\,dxdadt
\leq 2\Big(\int_{Q}|f|^{2}e^{2s\varphi}\,dxdadt
+\int_{Q}|\mu|^2|v|^{2}e^{2s\varphi}\,dxdadt\Big).
\end{aligned}
\end{equation}
Now, setting $\nu:=e^{s\varphi}v$,
we obtain
 \begin{equation}\label{sopra}
\begin{aligned}
 \int_{Q}|\mu|^2|v|^{2}e^{2s\varphi}\,dxdadt&\le  \|\mu\|_{\infty}^2\int_0^1
\nu^2 dx \\
&=  \|\mu\|_{\infty}^2\int_0^1\left(\frac{k^{1/3}}{|x-x_0|^{2/3}}\nu^2\right)^{3/4}\left(
\frac{|x-x_0|^2}{k} \nu^2\right)^{1/4} \\
&\le C \int_0^1
\frac{k^{1/3}}{|x-x_0|^{2/3}}\nu^2dx +C
\int_0^1 \frac{|x-x_0|^2}{k} \nu^2 dx.
\end{aligned}
\end{equation}
As in \eqref{quasfin}, proceeding as in \cite{fm1} and applying the Hardy-Poincar\'{e} inequality proved in \cite{fm} to the function $\nu$ with weight $p(x)= |x-x_0|^{4/3}$, if $K \le \displaystyle \frac{4}{3}$, or $p(x)  = (k(x)|x-x_0|^4)^{1/3}$, if $K>4/3$, we can prove that 

\begin{equation}\label{hpapplbis1}
\begin{aligned}
\int_0^1 \frac{ k^{1/3}}{|x-x_0|^{2/3}}\nu^2dx &\le C\int_0^1 k (\nu_x)^2 dx
\\
   &\le
    C\int_{Q} k(x) e^{2s\varphi}v_x^2dxdadt
 \\&  + Cs^2\int_{Q}\Theta^2 e^{2s\varphi}\frac{(x-x_0)^2}{k} v^2dxdadt.
\end{aligned}
\end{equation}
In any case, by \eqref{4'?}, \eqref{sopra} and \eqref{hpapplbis1}, we have
      \begin{equation}\label{fati2?}
      \begin{aligned}
      \int_{Q} |\bar{f}|^{2}\text{\small$~e^{2s\varphi}$\normalsize}~dxdadt
      &\le
      2\int_{Q} |f|^{2}\text{\small$~e^{2s\varphi}$\normalsize}~dxdadt
      +C\int_{Q}k(x) e^{2s\varphi} v_x^2 dxdadt
      \\&+ Cs^2\int_{Q} \Theta^2 e^{2s\varphi}\frac{(x-x_0)^2}{k} v^2dxdadt\\
      &\le  C\int_{Q} |f|^{2}\text{\small$~e^{2s\varphi}$\normalsize}~dxdadt
      +C\int_{Q}\Theta k(x) e^{2s\varphi} v_x^2 dxdadt
      \\&+ Cs^2\int_{Q} \Theta^3 e^{2s\varphi}\frac{(x-x_0)^2}{k} v^2dxdadt.
     \end{aligned} \end{equation}
   Substituting in
  \eqref{fati1?}, one can conclude
      \[
      \begin{aligned}
    &  \int_{Q}\left(s \Theta k v_x^2 +s^3\Theta^3\frac{(x-x_0)^2}{k} v^2\right)e^{2s\varphi}dxdadt
       \le
       C\Big(\int_{Q} |f|^{2}e^{2s\varphi}dxdadt
        \\[3pt]& 
         + s\int_0^T\int_0^A\left[k\Theta e^{2s\varphi}(x-x_0)(v_x)^2
dadt\right]_{x=0}^{x=1}dadt \Big),
        \end{aligned}
      \]
for all $s$ large enough.\\

\section{Observability and controllability}\label{osservabilita}
In this section we will prove, as a
consequence of the Carleman estimates established in Section 3, 
observability inequalities  for the associated  adjoint problem of \eqref{1}:
\begin{equation}\label{h=0}
\begin{cases}
\ds \frac{\partial v}{\partial t} + \frac{\partial v}{\partial a}
+(k(x)v_{x})_x-\mu(t, a, x)v +\beta(a,x)v(t,0,x)=0,& (t,x,a) \in  Q,
\\[5pt]
v(t,a,0)=v(t,a,1) =0, &(t,a) \in Q_{T,A},\\
  v(T,a,x) = v_T(a,x) \in L^2(Q_{A,1}), &(a,x) \in Q_{A,1} \\
  v(t,A,x)=0, & (t,x) \in Q_{T,1}.
\end{cases}
\end{equation}
 From now on, we assume that the control set $\omega$ is such that
\begin{equation}\label{omega0}
x_0 \in \omega =  (\alpha, \rho)  \subset\subset  (0,1),
\end{equation}
or
 \begin{equation}\label{omega_new}
 \omega = \omega_1 \cup
\omega_2,
\end{equation}
 where 
\begin{equation}\label{omega2}
\omega_i=(\lambda_i,\rho_i) \subset (0,1), \, i=1,2, \mbox{ and
$\rho_1 < x_0< \lambda_2$}.
\end{equation}

\begin{remark}\label{beta1}
Observe that, if \eqref{omega0} holds, we can find two subintervals
$\omega_1=(\lambda_1,\rho_1)\subset \subset(\alpha, x_0),
\omega_2=(\lambda_2,\rho_2) \subset\subset (x_0,\rho)$. 
\end{remark}

Moreover, on $\beta$ we assume the following assumption:
\begin{Assumptions}\label{conditionbeta} Suppose that there exists  
$\bar a <A$
 such that
\begin{equation}\label{conditionbeta1}
\beta(a, x)=0 \;  \text{for all $(a, x) \in [0, \bar a]\times [0,1]$}.
\end{equation}
\end{Assumptions} 

Observe that Hypothesis \ref{conditionbeta} has a biological meaning. Indeed, $\bar a$ is the minimal age in which the female of the population become fertile, thus it is natural that before $\bar a$ there are no newborns. For other comments on  Hypothesis \ref{conditionbeta} we refer to \cite{fJMPA}.

\vspace{0,4cm}
In order to prove the desired observability inequality 
for the solution $v$ of 
\eqref{h=0}
we proceed, as usual, using a density argument. To this purpose, we consider, first of all 
the space
\[
{\mathcal W}:=\Big\{ v\;\text{solution of \eqref{h=0}}\;\big|\;v_T \in
D(\mathcal A^2)\Big\}, \] where 
$D(\mathcal A^2) = \Big\{u \in
D(\mathcal A) \;\big|\; \mathcal A u \in
D(\mathcal A)\;\Big\}
$.  Clearly $D(\mathcal A^2)$ is densely
defined in $D({\mathcal A})$ (see, for example, \cite[Lemma 7.2]{b}) and
hence in $L^2(Q_{A,1})$ and 
\[\begin{aligned}
{\mathcal W}&= C^1\big([0,T]\:;D(\mathcal A)\big)\\&\subset \mathcal{V}:=L^2\big(Q_{T,A}; H^2_k(0,1)\big) \cap H^1\big(0,T; H^1(0,A;H^1_k(0,1))\big)
 \subset \mathcal{U}.
\end{aligned}\]

\begin{proposition}[Caccioppoli's inequality]\label{caccio} 
Let $\omega'$ and $\omega$ two open subintervals of $(0,1)$ such
that $\omega'\subset \subset \omega \subset  (0,1)$ and $x_0 \not
\in \bar\omega'$. Let $\psi(t,a,x):=\Theta(t,a)\Psi(x)$, where 
\begin{equation}\label{Theta}
\Theta(t, a)= \frac{1}{t^{4}(T-t)^{4}a^{4}}
\end{equation}
and $
\Psi \in C([0,1],(-\infty,0))\cap
C^1([0,1]\setminus\{x_0\},(-\infty,0))
$
is such that
\begin{equation}\label{stimayx}
|\Psi_x|\leq \frac{c}{\sqrt{k}} \mbox{ in }[0,1]\setminus\{x_0\}.
\end{equation}
Then, there exist two  strictly  positive constants $C$ and $s_0$ such that, for all $s \ge s_0$,
\begin{equation}\label{caccioeq}
\begin{aligned}
\int_{0}^T\int_0^A \int _{\omega'}   v_x^2e^{2s\psi } dxdadt
\ &\leq \ C\left( \int_{0}^T\int_0^A \int _{\omega}   v^2  dxdadt + \int_Q f^2 e^{2s\psi } dxdadt\right),
\end{aligned}
\end{equation}
for every
solution $v$ of \eqref{adjoint}.
\end{proposition}
The proof of the previous proposition is similar to the one given in \cite[Proposition 4.2]{f_anona} and \cite[Proposition 4.2]{fm}, so we omit it.

Moreover, the following non degenerate inequality proved in \cite{fJMPA} is crucial:
\begin{theorem}\label{nondegenere}[see \cite[Theorem 3.2]{fJMPA}]
Let $z\in 
\mathcal{Z}$ be the solution of
\eqref{adjoint},
where $f \in L^{2}(Q)$, $k \in C^{1}([0,1])$ is a strictly
positive function and
\[\mathcal{Z}:=L^2\big(Q_{T,A}; H^2(0,1)\cap H^1_0(0,1)\big) \cap H^1\big(0,T; H^1(0,A;H^1_0(0,1))\big).\]  Then, there exist two strictly positive constants $C$ and $s_0$,
such that, for any $s\geq s_0$, $z$ satisfies the estimate
\begin{equation} \label{570'}
\begin{aligned}
&\int_{Q}(s^{3}\phi^{3}z^{2}+s\phi z_{x}^{2})e^{2s\Phi} dxdadt \leq C \int_{Q}f^{2}e^{2s\Phi}dxdadt
  \\&-C
s\kappa\int_0^T\int_0^A\left[ke^{2s\Phi}\phi(z_x)^2
\right]_{x=0}^{x=1}dadt.
\end{aligned}
\end{equation}
Here the functions $\phi$ and $\Phi$ are
defined as follows
\begin{equation}\label{571}
\begin{gathered}
\phi(t,a,x)=\Theta(t,a)e^{\kappa\sigma(x)},\\
\Phi(a,t,x)=\Theta(t,a)\Psi(x), \quad
\Psi(x)=e^{\kappa\sigma(x)}-e^{2\kappa\|\sigma\|_{\infty}},
\end{gathered}
\end{equation}
where $(t,a,x)\in Q$, $\kappa>0$, $\sigma (x) :=\mathfrak{d}\int_x^1\frac{1}{k(t)}dt$,  $\fd=\|k'\|_{L^\infty(0,1)}$ and $\Theta$ is given in \eqref{Theta}.
\end{theorem}

\begin{remark}
The previous Theorem still holds under the weaker assumption $k \in W^{1, \infty}(0,1)$ without any additional assumption. 
\\
On the other hand, if we require $k \in W^{1,1}(0,1)$ then we have to add the following hypothesis:
{\it  there exist two functions $\fg \in L^1(0,1)$,
$\fh \in W^{1,\infty}(0,1)$ and two strictly positive constants
$\fg_0$, $\fh_0$ such that $\fg(x) \ge \fg_0$ and
\begin{equation}\label{debole}
-\frac{k'(x)}{2\sqrt{k(x)}}\left(\int_x^1\fg(t) dt + \fh_0 \right)+ \sqrt{k(x)}\fg(x) =\fh(x)\quad \text{for a.e.} \; x \in [0,1].
\end{equation}}
\\
In this case, i.e. if $k \in W^{1,1}(0,1)$,  the function $\Psi$ in \eqref{571} becomes
\begin{equation}\label{Psi_new}
\Psi(x):= - r\left[\int_0^x
\frac{1}{\sqrt{k(t)}} \int_t^1
\fg(s) dsdt + \int_0^x \frac{\fh_0}{\sqrt{k(t)}}dt\right] -\mathfrak{c}, 
\end{equation}
where $r$ and $\mathfrak{c}$ are suitable strictly positive functions. For other comments on Theorem \ref{nondegenere} we refer to \cite{fJMPA}.
\end{remark}
In the following, we will apply
 Theorem \ref{nondegenere} in the  intervals $[\lambda_2, 1]$ and $[-\rho_1, \rho_1]$ under these weaker assumptions. In particular, on $k$ we assume:

\begin{Assumptions}\label{ipogenerale}
The function $k$ satisfies Hypothesis $\ref{BAss01}$. Moreover, if  $k \in W^{1,1}(0,1)$, then
there exist
two functions $\fg \in L^\infty_{\rm loc}([-\rho_1,1]\setminus \{x_0\})$, $\fh \in W^{1,\infty}_{\rm loc}([-\rho_1,1]\setminus \{x_0\}, L^\infty(0,1))$ and
two strictly positive constants $\fg_0$, $\fh_0$ such that $\fg(x) \ge \fg_0$  and
\begin{equation}\label{aggiuntivastrana}
-\frac{\tilde k'(x)}{2\sqrt{\tilde k(x)}}\left(\int_x^B\fg(t) dt + \fh_0 \right)+ \sqrt{\tilde k(x)}\fg(x) =\fh(x,B)
\end{equation}
for a.e. $x \in [-\rho_1,1], B \in [0,1]$ with $x<B<x_0$ or $x_0<x<B$, where
\begin{equation}\label{tildek}
\tilde k(x):= \begin{cases}k(x), & x \in [0,1],\\
k(-x), & x \in [-1,0].
\end{cases}
\end{equation}
\end{Assumptions}
With the aid of Theorems \ref{Cor1},  \ref{nondegenere} and Proposition \ref{caccio}, we can now show $\omega-$local Carleman estimates  for \eqref{adjoint}.

\begin{theorem}\label{Cor2} Assume Hypothesis \ref{ipogenerale}. Then,
there exist two  strictly positive constants $C$ and $s_0$ such that every
solution $v$ of \eqref{adjoint} in
$
\mathcal {V}
$ 
satisfies, for all $s \ge s_0$,
\[
\begin{aligned}
\int_{Q}\left(s \Theta k  v_x^2
                + s^3\Theta^3\text{\small$\displaystyle \frac{(x-x_0)^2}{k}$\normalsize}
                  v^2\right)e^{2s\varphi}dxdadt
&\le
C\int_{Q}f^{2}\text{\small$e^{2s\Phi}$\normalsize}~dxdadt\\
&+C \int_0^T \int_0^A\int_ \omega v^2 dx dadt.
\end{aligned}\]
\end{theorem}
\begin{proof} First assume that $\omega$ satisfies \eqref{omega0} and take $w_i$, $i=1,2$, as in Remark \ref{beta1}. Now, fix $\bar \lambda_i, \bar \rho_i  \in \omega_i=(\lambda_i, \rho_i)$, $i=1,2$, such that $\bar \lambda_ i <  \bar \rho_i$ and  consider a smooth
function $\xi:[0,1]\to[0,1]$ such that
\[
\xi(x)=\begin{cases} 0&x\in [0,\bar\lambda_1],\\
1 & x\in[\tilde \lambda_1,\tilde \lambda_2],\\
0&x\in [\bar\rho_2,1],
\end{cases}
\]
where $\tilde \lambda_i=(\bar \lambda_i+\bar \rho_i)/2$, $i=1,2$. 
Define $w:= \xi v$, where $v$ is any fixed solution of
\eqref{adjoint}. 
    Then $w$   satisfies
    \[
    \begin{cases}
      w_t +w_a+(k  w_{x})_x- \mu w= \xi f + (k\xi_xv)_x+\xi_xkv_x=:h,&
      (t,a, x) \in Q,
      \\[5pt]
  w(t,a,0)=  w(t,a,1)=0, & (t,a) \in Q_{T,A}.
    \end{cases}
    \]
    Thus, applying Theorem \ref{Cor1}, Proposition \ref{caccio}, and proceeding as in \cite{f_anona}, we have
    \begin{equation}\label{add1}
    \begin{aligned}
&\int_0^T\int_0^A \int_{\tilde \lambda_1}^{\tilde \lambda_2}\left(s \Theta k v_x^2
                + s^3\Theta^3\text{\small$\displaystyle\frac{(x-x_0)^2}{k}$\normalsize}
                  v^2\right)e^{2s\varphi}dxdadt \\
                  &=\int_0^T\int_0^A \int_{\tilde \lambda_1}^{\tilde \lambda_2}\left(s \Theta k w_x^2
                + s^3\Theta^3\text{\small$\displaystyle\frac{x^2}{k}$\normalsize}
                  w^2\right)e^{2s\varphi}dxdadt\\
&
\le C \left( \int_{Q} f^2e^{2s\varphi} dxdadt+ \int_{0}^T\int_0^A \int _{\omega} v^2 dxdadt \right).
\end{aligned}
\end{equation}
Now, consider a smooth function $\eta: [0,1] \to [0,1]$ such that
\[
\eta(x) =\begin{cases} 0& x\in [0,\bar\lambda_2],\\
1& x\in [\tilde\lambda_2, 1],
\end{cases}
\]
and define $z:= \eta v$. Then $z$ satisfies
    \begin{equation}\label{problemz}
    \begin{cases}
      z_t +z_a+(k z_{x})_x- \mu z= \eta f +(k\eta_xv)_x+\eta_xkv_x=:h,&
      (t,a, x) \in Q_{T,A}\times (\lambda_2,1),\\
z(t,a, \lambda_2)= z(t,a, 1)=0, & t \in Q_{T,A},
\end{cases}
\end{equation}
    Clearly the equation satisfied by $z$ is not degenerate, thus applying Theorem \ref{nondegenere} and \cite[Lemma 4.1]{fm_opuscola} on $(\lambda_2,1)$, one has
    \[
    \begin{aligned}
&\int_0^T\int_0^A\int_{\lambda_2}^1(s^{3}\phi^{3}z^{2}+s\phi z_{x}^{2})e^{2s\Phi} dxdadt \leq C \int_0^T\int_0^A\int_{\lambda_2}^1h^{2}e^{2s\Phi}dxdadt
\\
 &\le C \left( \int_Qf^{2}e^{2s\Phi}dxdadt+  \int_0^T\int_0^A \int_{\omega}v^2dxdadt\right).
\end{aligned}
    \]
    Hence
    \[ 
      \begin{aligned}
  &  \int_0^T\int_0^A\int_{\tilde\lambda_2}^1   (s^{3}\phi^{3}v^{2}+s\phi v_{x}^{2})e^{2s\Phi} dxdadt=
 \int_0^T\int_0^A\int_{\tilde\lambda_2}^1 (s^{3}\phi^{3}z^{2}+s\phi z_{x}^{2})e^{2s\Phi} dxdadt\\
 &\le C \left( \int_Qf^{2}e^{2s\Phi}dxdadt+  \int_0^T\int_0^A \int_{\omega} v^2dxdadt\right),
 \end{aligned}
    \]
    for a  strictly positive constant $C$.
Proceeding, for example, as in \cite{fm1}  one can prove the existence of $\varsigma>0$, such that,  for
all $(t,a,x)\in [0,T]\times[0,A]\times[\lambda_2,1]$, we have
\begin{equation}\label{stimaphi}
e^{2s\varphi}\leq\varsigma e^{2s\Phi},
\frac{(x-x_0)^2}{k(x)}e^{2s\varphi}\leq\varsigma e^{2s\Phi}.
\end{equation}
Thus, for a  strictly  positive constant $C$,
   \begin{equation}\label{add2}
    \begin{aligned}
&\int_0^T\int_0^A \int_{\tilde\lambda_2}^1  \left(s \Theta k v_x^2
                + s^3\Theta^3\frac{(x-x_0)^2}{k}
                  v^2\right)e^{2s\varphi}dxdadt \\
                  &\le C\left( \int_0^T\int_0^A \int_{\tilde\lambda_2}^1  (s^{3}\phi^{3}v^{2}+s\phi v_{x}^{2})e^{2s\Phi} dxdadt\right)\\
                 &
\le C \left( \int_{Q} f^2 e^{2s\Phi} dxdadt+ \int_{0}^T\int_0^A \int _{\omega}  v^2  dxdadt \right).
\end{aligned}
\end{equation}
Hence,
\begin{equation}\label{stimavec}
  \begin{aligned}
&\int_0^T\int_0^A \int_{\tilde\lambda_1}^1  \left(s \Theta k v_x^2
                + s^3\Theta^3\frac{(x-x_0)^2}{k}
                  v^2\right)e^{2s\varphi}dxdadt \\
                  &
\le C \left( \int_{Q} f^2 e^{2s\Phi} dxdadt+ \int_{0}^T\int_0^A \int _{\omega}  v^2  dxdadt \right).
\end{aligned}
\end{equation}
To complete the
proof it is sufficient to prove a similar inequality for
$x\in[0,\tilde\lambda_1]$. To this aim, we use the reflection procedure as in
 \cite{fJMPA}; thus we consider the functions
\[
W(t,a,x):= \begin{cases} v(t,a,x), & x \in [0,1],\\
-v(t,a,-x), & x \in [-1,0],
\end{cases}
\]
\[
\tilde f(t,a,x):= \begin{cases} f(t,a,x), & x \in [0,1],\\
-f(t,a,-x), & x \in [-1,0],
\end{cases}
\]
\[
\tilde \mu(t,a,x):= \begin{cases} \mu(t,a,x), & x \in [0,1],\\
\mu(t,a,-x), & x \in [-1,0],
\end{cases}
\]
so that $W$ satisfies the problem
\[
\begin{cases}
W_t +W_a +(\tilde k W_{x})_x - \tilde \mu W= \tilde f, &(t,x) \in  Q_{T,A}\times (-1,1),\\
W(t,a,-1)=W(t,a,1) =0, & t \in Q_{T,A},
\end{cases}
\]
(by the way, observe that in \cite{fJMPA} there is a misprint in  the definition of $\mu$; it clearly must be defined in this way, otherwise $W$ is not the solution of the associated problem).
Now, consider a cut off function $\zeta: [-1,1] \to  [0,1]$ such that
\[\zeta (x) =\begin{cases} 0 & x\in[-1,-\bar \rho_1],\\
1& x\in [-\tilde\lambda_1, \tilde\lambda_1],\\
0&x\in [\bar \rho_1,1],
\end{cases}
\]
and define $Z:=\zeta W$. Then $Z$ satisfies
\begin{equation}\label{eq-Z*}
\begin{cases}
Z_t + Z_a+ (\tilde kZ_{x})_x -\tilde \mu Z=\tilde h,  &(t,x) \in Q_{T,A}\times (-\rho_1,\rho_1),\\
Z(t,a,-\rho_1)= Z(t,a,\rho_1)=0, & t \in Q_{T,A},
\end{cases}
\end{equation}
where $\tilde h=\zeta \tilde f+ (\tilde k\zeta_xW)_x+\zeta_x\tilde kW_x$.
Now, applying 
 the analogue of Theorem \ref{nondegenere} on $(- \rho_1,
\rho_1)$ in place of $(0,1)$, using the definition of $W$, the fact that $Z_x(t, a,-\rho_1)=Z_x(t,a,
\rho_1)=0$  and since $\zeta$ is
supported in $\left[-\bar \rho_1, -\tilde \lambda_1\right] \cup\left[\tilde\lambda_1, \bar \rho_1\right]$, we get
\[
\begin{aligned}
& \int_0^T\int_0^A\int_{0}^{\tilde \lambda_1}  \left(s\Theta k (W_x)^2 + s^3
\Theta^3
\frac{(x-x_0)^2}{k} W^2\right)e^{2s\varphi}dxdadt\\
&= \int_0^T\int_0^A\int_{0}^{\tilde \lambda_1}  \left(s\Theta k (Z_x)^2 + s^3
\Theta^3
\frac{(x-x_0)^2}{k} Z^2\right)e^{2s\varphi}dxdadt\\
&\le C \int_0^T\int_0^A\int_{0}^{\rho_1}\left(s\Theta (Z_x)^2 + s^3 \Theta^3
Z^2\right)e^{2s\Phi}dxdadt\\
&\le C \int_0^T\int_0^A\int_{-\rho_1}^{\rho_1}\left(s\Theta (Z_x)^2 + s^3 \Theta^3
Z^2\right)e^{2s\Phi}dxdadt
\end{aligned}
\]
\[
\begin{aligned}
& \le C
\int_0^T\int_0^A\int_{-\rho_1}^{\rho_1} \tilde h^{2}e^{2s\Phi}dxdadt\le C \int_0^T\int_0^A\int_{-\rho_1}^{\rho_1} \tilde f^{2}e^{2s\Phi}dxdadt \\
&+ C \int_0^T\int_0^A \int_{-\bar \rho_1}^{-\tilde\lambda_1}(
W^2+ (W_x)^2)e^{2s\Phi}dxdadt \\
&+ C\int_0^T\int_0^A \int_{\tilde \lambda_1}^{ \bar \rho_1}(W^2+ (W_x)^2)e^{2s\Phi}dxdadt\\
&\le C \int_0^T\int_0^A\int_{-\rho_1}^{\rho_1} \tilde f^{2}dxdadt + C\int_0^T\int_0^A \int_{-\rho_1}^{- \lambda_1}W^2dxdadt \\
&+C\int_0^T\int_0^A\int_{\lambda_1}^{ \rho_1} W^2dxdadt \\
& \mbox{ (by \cite[Lemma 4.1]{fm_opuscola} and since $\tilde f(t,a,x)= -f(t,a,-x)$, for $x <0$) }\\
& \le C \int_0^T\int_0^A\int_0^1 f^2 dxdadt +C\int_0^T\int_0^A\int_\omega v^2 dxdadt, 
\end{aligned}
\]
for some  strictly positive constants $C$ and $s$ large enough. Here  $\Phi$ is related to $  (-\rho_1,\rho_1)$.

 Hence, by definitions of $Z$, $W$ and $\zeta$, and using the previous inequality one has
\begin{equation}\label{car101}
\begin{aligned}
&\int_0^T\int_0^A\int_{0}^{\tilde\lambda_1} \left(s\Theta k (v_x)^2 + s^3 \Theta^3
\frac{(x-x_0)^2}{k} v^2\right)e^{2s\varphi}dxdadt\\
&= \int_0^T\int_0^A\int_{0}^{\tilde\lambda_1} \left(s\Theta k(W_x)^2 + s^3 \Theta^3
\frac{(x-x_0)^2}{k}W^2\right)e^{2s\varphi}dxdadt\\
&\le C\left( \int_Q f^{2}dxdadt + \int_0^T\int_0^A \int_{\omega}v^2dxdadt\right).
\end{aligned}
\end{equation}
Moreover, by \eqref{stimavec} and
\eqref{car101}, the conclusion follows.

Nothing changes in the proof if $\omega= \omega_1\cup \omega_2$ and each of these intervals lye on different sides of $x_0$, as the assumption implies.
\end{proof}
\begin{remark}\label{remarkultimo}
Observe that the results of Theorem \ref{Cor2}  still holds true if we substitute the domain $(0,T)\times (0,A)$ with a general domain $(T_1,T_2)\times (\delta,A)$, provided that $\mu$ and $\beta$ satisfy the required assumptions. In this case, in place of the function $\Theta$ defined in \eqref{Theta}, we have to consider the weight function
\[
\tilde \Theta(t,a):= \frac{1}{(t-T_1)^4 (T_2-t)^4(a-\delta)^4}.
\]
\end{remark}
Using the previous local Carleman estimates one can prove the next observability inequalities.

\begin{theorem}\label{Theorem4.4} Assume Hypotheses $\ref{conditionbeta}$, with $ \bar a<T \le A$, and $\ref{ipogenerale}$. Then,  for every $\delta \in (0,A)$,
there exists a  strictly positive constant $C=C(\delta)$  such that every
solution $v$ of \eqref{h=0} in
$\mathcal V$
satisfies
\begin{equation}\label{T<A}
\begin{aligned}
 &\int_0^A\int_0^1 v^2( T-\bar a,a,x) dxda  \le C\int_0^{T} \int_0^\delta \int_0^1v^2(t,a,x) dxdadt\\
 &+ C\left( \int_0^{T}\int_0^1 v_T^2(a,x)dxda+ \int_0^T \int_0^A\int_ \omega v^2 dx dadt\right).
\end{aligned}\end{equation}
Moreover, if $v_T(a,x)=0$ for all $(a,x) \in (0,  T) \times (0,1)$, one has
\begin{equation}\label{T<A1}
\begin{aligned}
 \int_0^A\!\!\int_0^1 v^2(T -\bar a,a,x) dxda  &\le C\int_0^{T} \int_0^\delta\!\! \int_0^1v^2(t,a,x) dxdadt \\&+C\int_0^T\!\! \int_0^A\!\!\int_ \omega v^2 dx dadt.
\end{aligned}
\end{equation}
\end{theorem}
Observe that in \cite[Theorem 4.4]{fJMPA}, which is the analogue of Theorem \ref{Theorem4.4} in the non divergence case, there is a mistake in the statement. Indeed, we assumed $\ds \frac{k'}{\sqrt{k}} \in L^\infty_{\text{loc}}([0,1]\setminus\{x_0\})$, which was a consequence of \eqref{aggiuntivastrana} below (see the remark after $(46)$ in \cite{fJMPA}); the precise assumption is:
\\
{\it there exist
two functions $\fg \in L^\infty_{\rm loc}([-\rho_1,1]\setminus \{x_0\})$, $\fh \in W^{1,\infty}_{\rm loc}([-\rho_1,1]\setminus \{x_0\}, L^\infty(0,1))$ and
two strictly positive constants $\fg_0$, $\fh_0$ such that $\fg(x) \ge \fg_0$  and
\begin{equation}\label{aggiuntivastrana1}
\frac{\tilde k'(x)}{2\sqrt{\tilde k(x)}}\left(\int_x^B\fg(t) dt + \fh_0 \right)+ \sqrt{\tilde k(x)}\fg(x) =\fh(x,B)
\end{equation}
for a.e. $x \in [-\rho_1,1], B \in [0,1]$ with $x<B<x_0$ or $x_0<x<B$, where $\tilde k$ is defined in \eqref{tildek}.
} Indeed, in order to prove \cite[Theorem 4.4]{fJMPA}, we use \cite[Theorem 4.3]{fJMPA} which  holds under \eqref{aggiuntivastrana1}. On the other hand, the statement of \cite[Corollary 4.1]{fJMPA}, which is also a consequence of \cite[Theorem 4.4]{fJMPA}, is correct.
\begin{proof}[Proof of Theorem \ref{Theorem4.4}] The proof follows the one of \cite[Theorem 4.4]{f_anona}, but we repeat here in a briefly way  for the reader's convenience underling the differences since in \cite[Theorem 4.4]{f_anona} $k$ degenerates at the boundary of the domain, while hereit degenerates in the interior. 

As in \cite{fJMPA}, using the method of characteristic lines, one can prove
the following implicit formula for $v$ solution of \eqref{h=0}:
\begin{equation}\label{implicitformula}
S(T-t) v_T(T+a-t, \cdot),
\end{equation}
if $t \ge  \tilde T + a$  and
\begin{equation}\label{implicitformula1}
v(t,a, \cdot)=\begin{cases}
S(T-t) v_T(T+a-t, \cdot)\!+\int_a^{T+a-t}S(s-a)\beta(s, \cdot)v(s+t-a, 0, \cdot) ds, &\Gamma\!= \!\bar a \\
\int_a^AS(s-a)\beta(s, \cdot)v(s+t-a, 0, \cdot) ds, & \Gamma\!= \!\Gamma_{A,T},
\end{cases}
\end{equation}
otherwise. Here  $(S(t))_{t \ge0}$ is the semigroup generated by the operator $\mathcal A_0 -\mu Id$ for all $u \in D(\mathcal A_0)$  ($Id$ is the identity operator), $\Gamma_{A,T}:= A -a +t-\tilde T$ and
\begin{equation}\label{Gamma}
\Gamma:= \min \{\bar a, \Gamma_{A,T}\}.
\end{equation}
In particular, it results
\begin{equation}\label{v(0)}
v(t,0, \cdot):= S(T-t) v_T(T-t, \cdot).
\end{equation}
Proceeding as in \cite[Theorem 4.4]{f_anona}, with suitable changes, one has that there exists a positive constant $C$ such that:
\begin{equation}\label{t=01}
 \int_{Q_{A,1}} v^2(\tilde T,a,x) dxda  \le C\int_{\frac{T}{4}}^{\frac{3T}{4}} \int_{Q_{A,1}} v^2(t,a,x) dxdadt.
\end{equation}
Take $\delta \in (0, A)$. By the previous inequality, we have
\begin{equation}\label{t=0}
\begin{aligned}
 \int_{Q_{A,1}} v^2(\tilde T,a,x) dxda  \le C \int_{\frac{T}{4}}^{\frac{3T}{4}} \left(\int_0^\delta + \int_\delta^A \right)\int_0^1 v^2(t,a,x) dxdadt.
\end{aligned}
\end{equation}
Now, we will estimate the term $\ds\int_{\frac{T}{4}}^{\frac{3T}{4}}  \int_\delta^A\int_0^1v^2(t,a,x) dxdadt$. It results that  
\begin{equation}\label{terminenuovo11}
\begin{aligned}
 \int_0^1v^2dx
& 
\le C\left( \int_0^1 k v_x^2 dx +\int_0^1 \frac{(x-x_0)^2}{k} v^2dx\right) ,
\end{aligned}\end{equation}
for a  strictly positive constant $C.$  
Indeed, using the Young's inequality to the function $v$,
we obtain
 \begin{equation}\label{nu'}
 \begin{aligned}
 \int_0^1|v|^{2}\,dx
                    & \le C
\int_0^1\left(\frac{k^{1/3}}{(x-x_0)^{2/3}}v^2\right)^{3/4}\left(
\frac{(x-x_0)^2}{k}v^2\right)^{1/4}dx \\
&\le C\int_0^1
\frac{k^{1/3}}{(x-x_0)^{2/3}}v^2dx+
 C \int_0^1
\frac{(x-x_0)^2}{k} v^2 dx.
 \end{aligned}
  \end{equation}
 Now, consider  the term 
  \[
  \int_0^1
\frac{k^{1/3}}{(x-x_0)^{2/3}}v^2dx.
 \]
 If $M > \ds \frac{4}{3}$,  take the function $\gamma(x) = (k(x)|x-x_0|^4)^{1/3}$. Clearly, $\displaystyle \gamma(x)=  k(x)
\left(\frac{(x-x_0)^2}{k(x)}\right)^{2/3}\le C k(x)$ and
$\displaystyle \frac{k^{1/3}}{(x-x_0)^{2/3}}=
\frac{\gamma(x)}{(x-x_0)^2}$. Moreover, using Hypothesis \ref{BAss01}, one
has that the function $\displaystyle\frac{\gamma(x)}{|x-x_0|^q} = \left(\frac{k(x)}{|x-x_0|^\theta}\right)^{\frac{1}{3}}$, where
$\displaystyle q: =\frac{4+\vartheta}{3}\in(1,2)$,  is non increasing on the left of $x=x_0$ and non decreasing on the right of $x=x_0$.  Hence,  by the Hardy-Poincar\'e inequality given in \cite[Proposition 2.6]{fm}, 
\[
 \int_0^1
\frac{k^{1/3}}{(x-x_0)^{2/3}}v^2dx =  \int_0^1
\frac{\gamma(x)}{(x-x_0)^2} v^2 dx  \le C \int_0^1 k v_x^2 dx.
\]
Thus, if $M > \ds\frac{4}{3}$, by \eqref{nu'},
\eqref{terminenuovo11} holds.
Now, assume  $M \le \ds\frac{4}{3}$  and introduce
the function $p(x) = |x-x_0|^{4/3}$. Obviously, there exists $ q \in
\left(1, \displaystyle\frac{4}{3}\right)$ such that the function
$\displaystyle x\mapsto\frac{p(x)}{|x-x_0|^q}$ is nonincreasing on
the left of $x=x_0$ and nondecreasing on the right of $x=x_0$. Thus,
applying again \cite[Proposition 2.6]{fm}, one has
\begin{equation}\label{hpapplbis}
\begin{aligned}
\int_0^1 \frac{k^{1/3}}{|x-x_0|^{2/3}}v^2dx & \le \max_{[0,1]}
k^{1/3}\int_0^1 \frac{1}{|x-x_0|^{2/3}}v^2dx \\
&= \max_{ [0,1]} k^{1/3}\int_0^1  \frac{p}{(x-x_0)^2} v^2 dx\\
&\le \max_{[0,1]} k^{1/3}C\int_0^1 p (v_x)^2 dx \\
& = \max_{[0,1]} k^{1/3} C\int_0^1  k \frac
{|x-x_0|^{4/3}}{k} (v_x)^2 dx\\
&\le \max_{ [0,1]}k^{1/3} C \int_0^1 k (v_x)^2 dx,
\end{aligned}
\end{equation}
Hence, \eqref{terminenuovo11} still holds and
\begin{equation}\label{ribo1}
\begin{aligned}
\int_{\frac{T}{4}}^{\frac{3T}{4}}  \int_\delta^A\int_0^1v^2(t,a,x) dxdadt &\le C \int_{\frac{T}{4}}^{\frac{3T}{4}}  \int_\delta^A\int_0^1\tilde \Theta v_x^2e^{2s\varphi} dxdadt \\
& +C \int_{\frac{T}{4}}^{\frac{3T}{4}}  \int_\delta^A\int_0^1 \tilde \Theta^3 \frac{(x-x_0)^2}{k}v^2e^{2s\varphi} dxdadt.
\end{aligned}
\end{equation}
The rest of the proof follows as in \cite[Theorem 4.4]{f_anona}, so we omit it.
 
\end{proof}

\begin{corollary}\label{CorOb} Assume Hypotheses $\ref{conditionbeta}$, with $\bar a=T <A$, and $\ref{ipogenerale}$.  Then, for every $\delta \in (0,A)$, there exists a  strictly positive constant $C=C(\delta)$  such that every
solution $v$ of \eqref{h=0} in
$\mathcal V$
satisfies
\[
\begin{aligned}
 \int_0^A\int_0^1 v^2( 0,a,x) dxda  &\le C\int_0^T \int_0^\delta \int_0^1v^2(t,a,x) dxdadt\\
 &+ C\left( \int_0^T\int_0^1 v_T^2(a,x)dxda+ \int_0^T \int_0^A\int_ \omega v^2 dx dadt\right).
\end{aligned}\]
Moreover, if $v_T(a,x)=0$ for all $(a,x) \in (0, T) \times (0,1)$, one has
\[
\begin{aligned}
 \int_0^A\!\!\int_0^1  v^2(0,a,x) dxda  &\le C\left(\int_0^T \int_0^\delta\!\! \int_0^1 v^2(t,a,x) dxdadt +\!\! \int_0^T\!\! \int_0^A\!\!\int_ \omega v^2 dx dadt\right).
\end{aligned}
\]
\end{corollary}

Proceeding as in Theorem \ref{Theorem4.4}, one can prove the analogous result in the case  $T>A$. Indeed, with suitable changes, one can prove again
\eqref{implicitformula},
if $t \ge  \tilde T + a$,  and
\eqref{implicitformula1},
otherwise. 
In particular, we have again
\eqref{v(0)}.  Thus:

\begin{theorem}\label{Theorem4.4_new} Assume Hypotheses $\ref{conditionbeta}$, with $ \bar a<A<T$, and $\ref{ipogenerale}$. Then, for every $\delta \in (0,A)$, there exists a  strictly positive constant $C=C(\delta)$  such that every
solution $v$ of \eqref{h=0} in
$\mathcal V$
satisfies
\eqref{T<A}.
Moreover, if $v_T(a,x)=0$ for all $(a,x) \in (0,  T) \times (0,1)$, one has
\eqref{T<A1}.
\end{theorem}

Actually, proceeding as in \cite{fJMPA} with suitable changes, we can improve the previous results in the following way:
\begin{theorem}\label{CorOb1'}Assume Hypotheses $\ref{conditionbeta}$ and $\ref{ipogenerale}$. If $T<A$, then, for every $\delta \in (T,A)$,
there exists a  strictly positive constant $C=C(\delta)$  such that every
solution $v$ of \eqref{h=0} in
$\mathcal V$
satisfies
\begin{equation}\label{ribo}
 \int_0^A\int_0^1 v^2(T-\bar a,a,x) dxda \le 
 C\left( \int_0^\delta \int_0^1 v_T^2(a,x)dxda+ \int_0^T \int_0^A\int_ \omega v^2 dx dadt\right).
\end{equation}
If $A<T$, then, for every $\delta \in (\bar a,A)$,
there exists a  strictly positive constant $C= C(\delta)$  such that every
solution $v$ of \eqref{h=0} in
$\mathcal V$
satisfies \eqref{ribo}.
\end{theorem}

\begin{proof}
If $T<A$ the proof of the previous theorem is analogous to the one of \cite[Theorem 4.6]{fJMPA}, with suitable changes, so we omit it.
\\
Now, consider the case $A<T$ and fix $\delta \in (\bar a, A)$. We distinguish between the two cases $\delta < 2\bar a$ and $\delta \ge 2\bar a$.

First of all,  consider  $\delta < 2\bar a$: as in \cite[Theorem 4.4.]{f_anona}, we can prove
\begin{equation}\label{bo}
 \int_{Q_{A,1}}  v^2(T-\bar a,a,x) dxda  \le C \int_{Q_{A,1}} v^2(t,a,x) dxda.
\end{equation}
Then, integrating over $\ds \left[T-\bar a,  T+\delta -2\bar a\right]$, we have the following inequality:
\begin{equation}\label{t=041}
 \int_{Q_{A,1}} v^2(T-\bar a,a,x) dxda  \le C\int_{T-\bar a}^{T+\delta -2\bar a} \left(\int_0^{\delta - \bar a} + \int_{\delta - \bar a}^A \right)\int_0^1 v^2(t,a,x) dxdadt.
\end{equation}
Using Theorem \ref{Cor2}, we can prove
\begin{equation}\label{bo1}
\begin{aligned}
\int_{T-\bar a}^{T+\delta -2\bar a}\int_{\delta - \bar a}^A \!\!\int_0^1v^2(t,a,x) dxdadt &\le C\int_0^\delta\int_0^1\!\! v_T^2(a,x)dxda\\&+C \int_0^T \!\!\int_0^A\!\!\int_ \omega v^2 dx dadt.
\end{aligned}
\end{equation}
Indeed, by \eqref{ribo1} applied to $[T-\bar a, +\delta -2\bar a]$ and Theorem \ref{Cor2}, we have
\[
\begin{aligned}
&\int_{T-\bar a}^{T+\delta -2\bar a}  \int_\delta^A\int_0^1v^2(t,a,x) dxdadt \le C \int_{T-\bar a}^{T+\delta -2\bar a}  \int_\delta^A\int_0^1\tilde \Theta v_x^2e^{2s\varphi} dxdadt \\
& +C \int_{T-\bar a}^{T+\delta -2\bar a} \int_\delta^A\int_0^1 \tilde \Theta^3 \frac{(x-x_0)^2}{k}v^2e^{2s\varphi} dxdadt\\
& \le 
C\left(\int_{T-\bar a}^{T+\delta -2\bar a} \int_0^A\int_0^1f^2dxdadt+ \int_0^T \int_0^A\int_ \omega v^2 dx dadt\right),
\end{aligned}
\]
where, in this case, $f(t,a,x):=-\beta(a,x)v(t,0,x)$. Hence,
\[
\begin{aligned}
&\int_{T-\bar a}^{T+\delta -2\bar a}  \int_\delta^A\int_0^1v^2(t,a,x) dxdadt  \le 
C\int_{T-\bar a}^{T+\delta -2\bar a} \int_0^A\int_0^1v^2(t,0,x)dxdadt\\&+ C\int_0^T \int_0^A\int_ \omega v^2 dx dadt\\
&\le C \left(\int_{T-\bar a}^{T+\delta -2\bar a} \int_0^1v_T^2(T-t,x)dxdt+ \int_0^T \int_0^A\int_ \omega v^2 dx dadt\right)\\
&=C \left(\int_{-\delta+2\bar a}^{\bar a}\int_0^1v_T^2(a,x)dxda+ \int_0^T \int_0^A\int_ \omega v^2 dx dadt\right)\\
&\le C  \left(\int_0^{\delta} \int_0^1v_T^2(a,x)dxda+ \int_0^T \int_0^A\int_ \omega v^2 dx dadt\right).
\end{aligned}
\]
Hence \eqref{bo1} follows.

It remains to estimate the following integral:
\[
\int_{T-\bar a}^{T+\delta -2\bar a}\int_0^{\delta - \bar a}\int_0^1  v^2(t,a,x) dxdadt.
\]
Observe that, since $t \le T+\delta -2\bar a$,  $t-T+\bar a < \delta- \bar a$, hence
\begin{equation}\label{zero}
\begin{aligned}
\int_{T-\bar a}^{T+\delta -2\bar a}\int_0^{\delta - \bar a}\int_0^1  v^2(t,a,x) dxdadt&= \int_{T-\bar a}^{T+\delta -2\bar a}\int_0^{t-T+\bar a}\int_0^1  v^2(t,a,x) dxdadt\\& + \int_{T-\bar a}^{T+\delta -2\bar a}\int_{t-T+\bar a}^{\delta -\bar a} \int_0^1  v^2(t,a,x) dxdadt.
\end{aligned}
\end{equation}
Now, by \eqref{implicitformula} and  by the boundedness of $(S(t))_{t \ge0}$,
\begin{equation}\label{prima}
\begin{aligned}
 &\int_{T-\bar a}^{T+\delta -2\bar a}\int_0^{t-T+\bar a}\int_0^1  v^2(t,a,x) dxdadt\\
 &=\int_{T-\bar a}^{T+\delta -2\bar a}\int_0^{t-T+\bar a}\int_0^1 (S(T-t)v_T(T+a-t,x))^2dxdadt\\
 & \le C\int_{T-\bar a}^{T+\delta -2\bar a}\int_0^{t-T+\bar a}\int_0^1 v_T^2(T+a-t,x)dxdadt\\
 &= C\int_{-\delta +2\bar a}^{\bar a}\int_0^{\bar a-z}\int_0^1 v_T^2(z+a,x)dxdadz\\
 &= C\int_{-\delta +2\bar a}^{\bar a}\int_z^{\bar a}\int_0^1 v_T^2(\sigma,x)dxd\sigma dz\\
 &\le C\int_{-\delta +2\bar a}^{\bar a} \int_0^{\bar a}\int_0^1 v_T^2(\sigma,x)dxd\sigma dz\\
 &\le C \int_0^{\bar a}\int_0^1 v_T^2(\sigma,x)dxd\sigma \le C\int_0^\delta\int_0^1 v_T^2(\sigma,x)dxd\sigma
\end{aligned}
\end{equation}
On the other hand, if $t \ge T-\bar a$ and $a \in (t-T+\bar a, \delta -\bar a)$, it results that $T-t < A-a$, thus $\Gamma = \bar a$ (to this purpose recall that $\delta \in (\bar a, A)$ and $\Gamma$ is defined in \eqref{Gamma}). Hence in \eqref{implicitformula1} we have to consider the first formula, i.e.
\[
v(t,a, \cdot)=
S(T-t) v_T(T+a-t, \cdot)\!+\int_a^{T+a-t}S(s-a)\beta(s, \cdot)v(s+t-a, 0, \cdot) ds.
\]
It follows that, proceeding as in \eqref{prima},
\[
\begin{aligned}
&
\int_{T-\bar a}^{T+\delta -2\bar a}\int_{t-T+\bar a}^{\delta -\bar a} \int_0^1  v^2(t,a,x) dxdadt\\
&\le C\int_{T-\bar a}^{T+\delta -2\bar a}\int_{t-T+\bar a}^{\delta -\bar a} \int_0^1 v^2_T(T+a-t,x)dxdadt \\
&+C \int_{T-\bar a}^{T+\delta -2\bar a}\int_{t-T+\bar a}^{\delta -\bar a} \int_0^1  \left(\int_a^{T+a-t}v^2(s+t-a,0,x)ds \right)dxdadt\\
&= C\int_{-\delta+2 \bar a}^{\bar a}\int_{\bar a-z}^{\delta -\bar a} \int_0^1 v^2_T(z+a,x)dxdadz \\
& +C \int_{T-\bar a}^{T+\delta -2\bar a}\int_{t-T+\bar a}^{\delta -\bar a} \int_0^1  \left(\int_a^{T+a-t}v_T^2(T-s-t+a,x)ds \right)dxdadt
\\
&=C\int_{-\delta+2 \bar a}^{\bar a}\int_{\bar a}^{z+\delta -\bar a} \int_0^1 v^2_T(\sigma,x)dxd\sigma dz \\
&+C \int_{T-\bar a}^{T+\delta -2\bar a}\int_{t-T+\bar a}^{\delta -\bar a} \int_0^1  \left(\int_{-a}^{T-a-t}v_T^2(a+z,x)dz \right)dxdadt
\end{aligned}
\]
\[
\begin{aligned}
& \le C\int_{-\delta+2 \bar a}^{\bar a}\int_{\bar a}^{z+\delta -\bar a} \int_0^1 v^2_T(\sigma,x)dxd\sigma dz \\
&+C \int_{T-\bar a}^{T+\delta -2\bar a}\int_{t-T+\bar a}^{\delta -\bar a} \int_0^1  \left(\int_{-a}^{T-a-t}v_T^2(a+z,x)dz \right)dxdadt.
\end{aligned}
\]
Using the fact that in the first integral $z \in (-\delta+2 \bar a, \bar a)$ and in the second one  $t \ge T-\bar a$, one has $z+\delta-\bar a \le \delta$ and  $T-t \le \bar a$, respectively, this implies
\begin{equation}\label{seconda}
\begin{aligned}
&\int_{T-\bar a}^{T+\delta -2\bar a}\int_{t-T+\bar a}^{\delta -\bar a} \int_0^1  v^2(t,a,x) dxdadt\le  C\int_{-\delta+2 \bar a}^{\bar a}\int_{\bar a}^{\delta} \int_0^1 v^2_T(\sigma,x)dxd\sigma dz \\
&+C \int_{T-\bar a}^{T+\delta -2\bar a}\int_{t-T+\bar a}^{\delta -\bar a} \int_0^1  \left(\int_0^{T-t}v_T^2(\sigma, x)d\sigma \right)dxdadt\\
&\le C \int_0^{\delta} \int_0^1 v^2_T(\sigma,x)dxd\sigma + C\int_{T-\bar a}^{T+\delta -2\bar a}\int_{t-T+\bar a}^{\delta -\bar a} \int_0^1  \left(\int_0^{\bar a}v_T^2(\sigma, x)d\sigma \right)dxdadt\\
&\le C \int_0^{\delta} \int_0^1 v^2_T(\sigma,x)dxd\sigma.
\end{aligned}
\end{equation}
Hence, by \eqref{t=041} - \eqref{seconda}, \eqref{ribo} follows.

Now, consider the case $\delta \ge 2\bar a$ and, in place of $\ds \left[T-\bar a,  T+\delta -2\bar a\right]$, take the interval $\left[T-\bar a, T-\ds\frac{\bar a}{4}\right]$. Hence we have
\begin{equation}\label{t=041new}
 \int_{Q_{A,1}} v^2(T-\bar a,a,x) dxda  \le C\int_{T-\bar a}^{T-\frac{\bar a}{4}} \left(\int_0^{\delta - \bar a} + \int_{\delta - \bar a}^A \right)\int_0^1 v^2(t,a,x) dxdadt.
\end{equation}
Proceeding as before, we can prove the analogous of \eqref{bo1}, i.e.
\begin{equation}\label{bo2}
\int_{T-\bar a}^{T-\frac{\bar a}{4}}\int_{\delta - \bar a}^A \!\!\int_0^1v^2(t,a,x) dxdadt \le C\left(\int_0^\delta\int_0^1\!\! v_T^2(a,x)dxda+ \int_0^T \!\!\int_0^A\!\!\int_ \omega v^2 dx dadt\right).
\end{equation}
 It remains to estimate
\[
\int_{T-\bar a}^{T-\frac{\bar a}{4}}\int_0^{\delta - \bar a}\int_0^1 v^2(t,a,x) dxdadt.
\]
Also in this case, since $t \in \left[T-\bar a, T-\ds\frac{\bar a}{4}\right]$, it follows that $t-T+\bar a \le \delta -\bar a$ (recall that we are in the case $\delta \ge 2\bar a$). Proceeding as before, one has 
\[
\begin{aligned}
\int_{T-\bar a}^{T-\frac{\bar a}{4}}\int_0^{\delta - \bar a}\int_0^1  v^2(t,a,x) dxdadt&= \int_{T-\bar a}^{T-\frac{\bar a}{4}}\int_0^{t-T+\bar a}\int_0^1  v^2(t,a,x) dxdadt\\& + \int_{T-\bar a}^{T-\frac{\bar a}{4}}\int_{t-T+\bar a}^{\delta -\bar a} \int_0^1  v^2(t,a,x) dxdadt.
\end{aligned}
\]
As for \eqref{prima} and \eqref{seconda}, we have:
\begin{equation}\label{primanew}
\begin{aligned}
& \int_{T-\bar a}^{T-\frac{\bar a}{4}}\int_0^{t-T+\bar a}\int_0^1  v^2(t,a,x) dxdadt\\
&
  \le C\int_{T-\bar a}^{T-\frac{\bar a}{4}}\int_0^{t-T+\bar a}\int_0^1 v_T^2(T+a-t,x)dxdadt\\
 &= C\int_{\frac{\bar a}{4}}^{\bar a}\int_0^{\bar a-z}\int_0^1 v_T^2(z+a,x)dxdadz= C\int_{\frac{\bar a}{4}}^{\bar a}\int_z^{\bar a}\int_0^1 v_T^2(\sigma,x)dxd\sigma dz\\
 &\le C\int_{\frac{\bar a}{4}}^{\bar a} \int_0^{\bar a}\int_0^1 v_T^2(\sigma,x)dxd\sigma dz \le C\int_0^\delta\int_0^1 v_T^2(\sigma,x)dxd\sigma
\end{aligned}
\end{equation}
and
\begin{equation}\label{secondanew}
\begin{aligned}
&
\int_{T-\bar a}^{T-\frac{\bar a}{4}}\int_{t-T+\bar a}^{\delta -\bar a} \int_0^1  v^2(t,a,x) dxdadt\\
&\le C\int_{T-\bar a}^{T-\frac{\bar a}{4}}\int_{t-T+\bar a}^{\delta -\bar a} \int_0^1 v^2_T(T+a-t,x)dxdadt \\
&+C \int_{T-\bar a}^{T-\frac{\bar a}{4}}\int_{t-T+\bar a}^{\delta -\bar a} \int_0^1  \left(\int_a^{T+a-t}v^2(s+t-a,0,x)ds \right)dxdadt\\
&=C\int_{\frac{\bar a}{4}}^{\bar a}\int_{\bar a}^{z+\delta -\bar a} \int_0^1 v^2_T(\sigma,x)dxd\sigma dz \\
&+C \int_{T-\bar a}^{T-\frac{\bar a}{4}}\int_{t-T+\bar a}^{\delta -\bar a} \int_0^1  \left(\int_{-a}^{T-a-t}v_T^2(a+z,x)dz \right)dxdadt\\
& \le C\int_{\frac{\bar a}{4}}^{\bar a}\int_{\bar a}^{\delta} \int_0^1 v^2_T(\sigma,x)dxd\sigma dz \\
&+C \int_{T-\bar a}^{T-\frac{\bar a}{4}}\int_{t-T+\bar a}^{\delta -\bar a} \int_0^1  \left(\int_0^{T-t}v_T^2(\sigma,x)d\sigma \right)dxdadt\\
& \le C\int_{0}^{\delta} \int_0^1 v^2_T(\sigma,x)dxd\sigma +C \int_{T-\bar a}^{T-\frac{\bar a}{4}}\int_{t-T+\bar a}^{\delta -\bar a} \int_0^1  \left(\int_0^{\bar a}v_T^2(\sigma,x)d\sigma \right)dxdadt\\
& \le C\int_0^\delta\int_0^1v_T^2(\sigma,x)d\sigma dx.
\end{aligned}
\end{equation}
By \eqref{t=041new}-\eqref{secondanew}, \eqref{ribo} follows.
\end{proof}

\vspace{0.4cm}
By Theorem \ref{CorOb1'} and using a density argument, one can deduce
the following observability result:
\begin{proposition}\label{PropOI}
\label{obser.} Assume Hypotheses $\ref{conditionbeta}$ and $\ref{ipogenerale}$. If $T<A$, then, for every $\delta \in (T,A)$,
there exists a  strictly positive constant $C=C(\delta)$  such that every
solution  $v\in \mathcal U$ of \eqref{h=0}
satisfies
\begin{equation}\label{OI}
 \int_0^A\int_0^1 v^2(T-\bar a,a,x) dxda \le 
 C\left( \int_0^\delta \int_0^1v_T^2(a,x)dxda+ \int_0^T \int_0^A\int_ \omega v^2 dx dadt\right).
\end{equation}
If $A<T$, then, for every $\delta \in (\bar a,A)$,
there exists a  strictly positive constant $C= C(\delta)$  such that every
solution $v$ of \eqref{h=0} satisfies \eqref{OI}.\\
Here $v_T(a,x)$ is such that $v_T(A,x)=0$ in $(0,1)$.
\end{proposition}

Observe that in the statements of the analogous results given in \cite{fJMPA} for the non divergence case there is a misprint. Indeed the constant $C$ depends on $\delta$, as one can deduce by the proofs.  The right statement is 
\begin{center}
{\it...for every $\delta \in (T,A),$ there exists $C=C(\delta)$ such that...}
\end{center}
We underline that the results are correct and in the correct way they are used to prove  \cite[Theorems 4.7 and  4.8]{fJMPA}.

\vspace{0.2cm}

 As a consequence of Proposition \ref{PropOI} one can prove, as in \cite[Theorem 4.7]{f_anona}, 
  the following null controllability result:

\begin{theorem}\label{ultimo} Assume Hypotheses  $\ref{conditionbeta}$ and $\ref{ipogenerale}$ and take $y_0 \in L^2(Q_{A,1})$. Then
 for every $\delta \in (T,A)$, if $T<A$, or for every $\delta \in (\bar a,A)$, if $A<T$, there exists a control  $f_\delta \in L^2(Q)$ such that the solution $y_\delta$ of  \eqref{1}
satisfies
\begin{equation}\label{chissa}
y_\delta(T,a,x) =0 \quad \text{a.e. } (a,x) \in (\delta, A) \times (0,1).
\end{equation}
Moreover, there exists $C=C(\delta)>0$ 
\begin{equation}\label{stimaf}
\|f_\delta\|_{L^2( Q)} \le C \|y_0\|_{L^2(Q_{A,1})}.
\end{equation}
\end{theorem}

\section*{Acknowledgments}

The author is a member of the Gruppo Nazionale per l'Analisi Matematica, la Probabilit\`a e le loro Applicazioni (GNAMPA) of the
Istituto Nazionale di Alta Matematica (INdAM) and she is partially  supported by the FFABR ``Fondo
per il finanziamento delle attivit\`a base di ricerca'' 2017.

\end{document}